\documentclass{amsart}

\usepackage[all]{xy}
\usepackage{amsfonts}
\usepackage{amssymb}

\numberwithin{equation}{section}

\theoremstyle{plain}

\newtheorem{theorem}[subsection]{Theorem}
\newtheorem{conjecture}[subsection]{Conjecture}
\newtheorem{proposition}[subsection]{Proposition}
\newtheorem{lemma}[subsection]{Lemma}
\newtheorem{corollary}[subsection]{Corollary}

\newtheorem{case}{Case}

\theoremstyle{definition}

\renewcommand{\leq}{\leqslant}
\renewcommand{\geq}{\geqslant}

\renewcommand{\P}{\mathbb{P}}
\renewcommand{\Im}{\mathop{\rm Im}\nolimits}

\newcommand{\wh}{\widehat}
\newcommand{\E}{\mathbb{E}}
\newcommand{\Z}{\mathbb{Z}}

\newcommand{\C}{\mathbb{C}}
\newcommand{\F}{\mathbb{F}}

\begin{document}

\title{Roth's theorem in $\Z_4^n$}

\author{T. Sanders}
\address{Department of Pure Mathematics and Mathematical Statistics\\
University of Cambridge\\
Wilberforce Road\\
Cambridge CB3 0WA\\
England } \email{t.sanders@dpmms.cam.ac.uk}

\begin{abstract}
We show that if $A \subset \Z_4^n$ contains no three-term arithmetic progressions in which all the elements are distinct then $|A|=o(4^n/n)$.
\end{abstract}

\maketitle

\section{Introduction}

Suppose that $G$ is a finite abelian group. A three-term arithmetic progression in $G$ is a triple $(x,x+d,x+2.d)$ with $x,d \in G$; a proper progression is one in which all the elements are different i.e. $2.d \neq 0_G$.

In \cite{KFRPre,KFR} Roth famously proved that any subset of $\Z/N\Z$ of sufficiently large density contains a proper three-term arithmetic progression, a result which was generalised by Meshulam in \cite{RM}.
\begin{theorem}[Meshulam]\label{thm.meshulam}
Suppose that $G$ is a finite abelian group of odd order and $A \subset G$ contains no proper three-term arithmetic progressions. Then 
\begin{equation*}
|A|=O(|G|/\log^{\Omega(1)}|G|).
\end{equation*}
\end{theorem}
An explicit value for the $\Omega(1)$ constant can be read out of the proof, and it seems that in light of the recent work of Bourgain \cite{JBRoth2} (itself improving on \cite{JB,ES} and \cite{DRHB}) one could probably take any constant strictly less than $2/3$. While this appears to be the limit in general, for certain groups one can do better. Indeed, for\footnote{Or, more generally, any abelian group of odd order and bounded exponent.} $\Z_3^n$ Roth's original argument simplifies considerably to give the following result which is qualitatively due to Brown and Buhler \cite{TCBJPB}.
\begin{theorem}[Roth-Meshulam] Suppose that $G=\Z_3^n$ and $A \subset G$ contains no proper three-term arithmetic progressions. Then
\begin{equation*}
|A|=O(|G|/\log |G|).
\end{equation*}
\end{theorem}
The question of what the true bounds on $|A|$ are arises in many different studies (see \cite{PMRLGVR,SYID,YE} and \cite{YECEAGSKLR}) and improving the bound is a well known open problem as reported in \cite{BJGFFM,ESCVFL,TCTBlog07}; the closest anyone has come is in \cite{ESC3,ESC2}. While we are not able to make progress on this question it is the purpose of this paper to show an improvement for a different class of groups.

It was quite natural in Theorem \ref{thm.meshulam} to insist that $G$ be of odd order: in the group $\Z_2^n$ every arithmetic progression is easily seen to be of the form $(x,y,x)$, so no set contains a proper progression. Not all groups of even order are as trivial as $\Z_2^n$ and, as part of a more general corpus of results, Lev resolved the question of which abelian groups Meshulam's theorem could be extended to in \cite{VFL}.
\begin{theorem}[Lev]\label{thm.lev}
Suppose that $G=\Z_4^n$ and $A \subset G$ contains no proper three-term arithmetic progressions. Then
\begin{equation*}
|A|=O(|G|/\log |G|).
\end{equation*}
\end{theorem}
The above special case of Lev's work follows rather easily from the method used to prove the Roth-Meshulam theorem coupled with a positivity observation. At considerable further expense we are able to establish the following minor improvement.
\begin{theorem}\label{thm.maintheorem}
Suppose that $G=\Z_4^n$ and $A \subset G$ contains no proper three-term arithmetic progressions. Then
\begin{equation*}
|A|=O(|G|/\log |G|\log \log^{\Omega(1)} |G|).
\end{equation*}
\end{theorem}
It should be noted that the requirement that \emph{all} the elements of our progressions be distinct is essential in our work. It is easy to see by the Cauchy-Schwarz inequality that any set $A \subset G:= \Z_4^n$ has at least $\alpha^2|G|^{3/2}$ progressions. It follows that if $\alpha^2|G|^{3/2}>|G|$ then $A$ contains a progression in which not all the elements are the same, however this may well be a degenerate one of the form $(x,y,x)$.

The paper now splits as follows. First, in \S\ref{sec.ft}, we record the necessary information about the Fourier transform. Then, in \S\S\ref{sec.anal}\verb!&!\ref{sec.out}, we outline our approach to counting progressions and compare it with the Roth-Meshulam-Lev method to give some indication of where we are able to make gains. In \S\ref{sec.families} we define the notion of a family which we shall work with for the bulk of the paper and the proof of Theorem \ref{thm.maintheorem} in \S\S\ref{sec.dense}--\ref{sec.end}.  We close in \S\ref{sec.obstacles} with a conjecture some concluding remarks on lower bounds.

\section{The {F}ourier transform}\label{sec.ft}

We shall make considerable use of the Fourier transform for which the classic book \cite{WR} of Rudin serves as the standard reference. Having said this, the style of our work has more in common with the modern reference \cite{TCTVHV} of Tao and Vu which is also to be recommended.

Suppose that $G$ is a finite abelian group. $\wh{G}$ denotes the \emph{dual group} of $G$, that is the group of homomorphisms $\gamma:G\rightarrow S^1$, where $S^1:=\{z \in \C:|z|=1\}$.  $G$ is endowed with a natural Haar probability measure denoted $\P_G$ assigning mass $|G|^{-1}$ to each element of $G$; we denote integration against $\P_G$ by $\E_{x \in G}$ and, in general $\E_{x \in S}$ corresponds to integration against the probability measure $\P_S$ assigning mass $|S|^{-1}$ to each $s \in S$. 

For $p \in [1,\infty]$ we define the spaces $L^p(G)$ and $\ell^p(G)$ to be the vector space of functions $f:G \rightarrow \C$ endowed with the norms
\begin{equation*}
\|f\|_{L^p(G)}:=\left(\E_{x \in G}{|f(x)|^p}\right)^{1/p} \textrm{ and } \|f\|_{\ell^p(G)}:=\left(\sum_{x \in G}{|f(x)|^p}\right)^{1/p}
\end{equation*}
respectively, with the usual conventions when $p=\infty$. As vector spaces these are all the same (since $G$ is finite), although the norms are different.  A specific consequence of this normalisation is that
\begin{equation*}
\langle f,g \rangle_{L^2(G)} = \E_{x \in G}{f(x)\overline{g(x)}} \textrm{ and } \langle f,g\rangle_{\ell^2(G)}=\sum_{x \in G}{f(x)\overline{g(x)}}.
\end{equation*}

We define the Fourier transform in the usual way, mapping a function $f \in L^1(G)$ to $\wh{f} \in \ell^\infty(\wh{G})$ where
\begin{equation*}
\wh{f}(\gamma):=\E_{x \in G}{f(x)\overline{\gamma(x)}}=\frac{1}{|G|}\sum_{x \in G}{f(x)\overline{\gamma(x)}}.
\end{equation*}
The significance of the Fourier transform is, in no small part, determined by the effect it has on convolution: recall that if $f,g \in L^1(G)$ then their convolution $f \ast g$ is defined by
\begin{equation*}
(f \ast g)(x):=\E_{x \in G}{f(x)g(y-x)}.
\end{equation*}
The Fourier transform functions as an algebra isomorphism from $L^1(G)$ under convolution to $\ell^\infty(\wh{G})$ under point-wise multiplication: $\wh{f \ast g}=\wh{f} . \wh{g}$.

As a word of warning we remark that both convolution and the Fourier transform are used on different groups at the same time through this work, and although it is always made clear, the reader should be alert to this.

We are particularly interested in finite (abelian) groups of exponent $2$, all of which are isomorphic to $\Z_2^n$ for some $n$; to avoid introducing an unnecessary parameter we shall refer to them in the former terms. On these groups the characters correspond to maps $x \mapsto (-1)^{r.x}$, where $r.x$ is the usual bilinear form on $\Z_2^n$ considered as a vector space over $\F_2$.

\section{Counting progressions and analytic statement of results}\label{sec.anal}

It has been observed in many places that one may estimate the size of the largest subset of an abelian group not containing a three-term arithmetic progression by establishing a lower bound on the number of three-term arithmetic progressions.  It should, therefore, come as little surprise that we are interested in the quantity
\begin{equation*}
\Lambda(A):=\E_{x,d \in G}{1_A(x)1_A(x+d)1_A(x+2.d)}.
\end{equation*}
which counts three-term arithmetic progressions: specifically $\Lambda(A)|G|^2$ is the number of three-term arithmetic progressions in $A$. 

Denoting by $T(G)$ the number of trivial\footnote{i.e. not proper.}  three-term arithmetic progressions in $G$ we see that if $\Lambda(A)|G|^2 > T(G)$ then we must have a non-trivial three-term arithmetic progression.  This perspective is, perhaps, inspired by Varnavides who established an equivalence in \cite{PV} but we shall not dwell on this relationship here.

Meshulam's theorem is a simple corollary of the following result.
\begin{theorem}\label{thm.meshcount}
Suppose that $G$ is a finite abelian group of odd order and $A \subset G$ has density $\alpha>0$. Then \begin{equation*}
\Lambda(A) \geq \exp(-\alpha^{-O(1)}).
\end{equation*}
\end{theorem}
To see how Meshulam's theorem follows, note that if $G$ is of odd order then $(x,x+d,x+2.d)$ is a proper progression if and only if $d \neq 0_G$. Thus $T(G) =|G|$ and so if $\Lambda(A)|G|^2 > |G|$ then $A$ contains a proper progression; the result follows on inserting the bound for $\Lambda(A)$ from the theorem and rearranging.

In \cite{VFL} Lev effectively removed the `odd order condition' from Theorem \ref{thm.meshcount} to establish the following result.
\begin{theorem}\label{thm.lef}
Suppose that $G$ is a finite abelian group and $A \subset G$ has density $\alpha>0$. Then\begin{equation*}
\Lambda(A) \geq \exp(-\alpha^{-O(1)}).
\end{equation*}
\end{theorem}
In general abelian groups $T(G)$ may be comparable to $|G|^2$ which is why we are not able to conclude  Meshulam's theorem without the odd order condition.  Indeed, as noted before it is not always true.

It is instructive to consider two examples.  First, in $G=\Z_2^n$ one sees that $T(G)=|G|^2$ -- all progressions are trivial -- so although we have many\footnote{In fact it is easy to see this without Theorem \ref{thm.lef}: $A \subset \Z_2^n$ clearly contains $|A|^2$ progressions since every pair $(x,y) \in A^2$ generates a triple $(x,y,x)$ which is a three-term arithmetic progression in $\Z_2^n$.} progressions, none are proper.

Second, the group $G=\Z_4^n$ has $T(G)=|G|^{3/2}+O(|G|)$: any trivial progression $(x,y,z)$ with $x+z=2.y$ has $x=z$, $x=y$ or $y=z$.  In the first case this implies that $x-y \in \{x' \in G: 2.x'=0_G\}$; in the second and third cases this implies that all three elements are equal.  Thus, in the first case we have $|G|. |\{x' \in G: 2.x'=0_G\}|$ progressions and in the second and third $|G|$ each.  This leads to the claimed bound which in turn allows us to establish Meshulam's theorem for $\Z_4^n$.

In this particular case, however, one may proceed directly along the lines of the proof of the Roth-Meshulam theorem (coupled with the aforementioned positivity observation) to establish a stronger bound than in Theorem \ref{thm.lef}.
\begin{theorem}\label{thm.weak}
Suppose that $G=\Z_4^n$ and $A \subset G$ has density $\alpha>0$. Then
\begin{equation*}
\Lambda(A) \geq \exp(-O(\alpha^{-1})).
\end{equation*}
\end{theorem}
On arranging $\alpha$ large enough so that $\Lambda(A)|G|^2 > |G|^{3/2}+O(|G|)$ is guaranteed by the above theorem we get Theorem \ref{thm.lev}; the main result of this paper is the following refinement of Theorem \ref{thm.weak} which by a similar arrangement implies Theorem \ref{thm.maintheorem}.
\begin{theorem}\label{thm.weighted}
Suppose that $G=\Z_4^n$ and $A \subset G$ has density $\alpha>0$. Then 
\begin{equation*}
\Lambda(A) \geq \exp(-O(\alpha^{-1}\log^{-1/6}\alpha^{-1}\log \log ^{5/3}\alpha^{-1})).
\end{equation*}
\end{theorem}

\section{Outline of the proof}\label{sec.out}

Our work is strongly influenced by the original Roth-Meshulam-Lev argument; to explain our extra purchase we shall recall a sketch of this. There are basically three ingredients.  First, one has a lemma passing from a large Fourier coefficient to increased density on a subgroup.
\begin{lemma}\label{lem.rmit}
Suppose that $G$ is a group of bounded exponent, $A \subset G$ has density $\alpha>0$ and $\sup_{\gamma \neq 0_{\wh{G}}}{|\wh{1_A}(\gamma)|} \geq \epsilon\alpha$. Then there is a subgroup $G' \leq G$ of bounded index such that $\|1_A \ast \P_{G'}\|_{L^\infty(G)} \geq \alpha+\Omega(\alpha \epsilon)$.
\end{lemma}
The proof of this is easy and we shall use some similar results in \S\ref{sec.dense}; we make no improvement on this ingredient and, indeed, the lemma is in many ways best possible.

The core of the argument is the following lemma and it is here that we shall do better. The lemma expresses the fact that either a set $A$ is `uniform' having about the right number of three-term arithmetic progressions or else it has increased density on a subgroup of bounded index.
\begin{lemma}\label{lem.itrm}
Suppose that $G$ is a group of bounded exponent and $A \subset G$ has density $\alpha>0$. Then either $\Lambda(A) =\Omega(\alpha^3)$ or there is a subgroup $G' \leq G$ of bounded index such that $\|1_A \ast \P_{G'}\|_{L^\infty(G)} \geq \alpha+\Omega(\alpha^2)$.
\end{lemma}
\begin{proof}[Proof (sketch)]
By the usual application of the inversion formula one has
\begin{equation*}
\Lambda(A) = \sum_{\gamma \in \wh{G}}{\wh{1_A}(\gamma)^2\wh{1_A}(2.\gamma)}.
\end{equation*}
We write $H:=\{\gamma \in \wh{G}:2.\gamma = 0_{\wh{G}}\}$ so that 
\begin{equation*}
\Lambda(A) = \alpha.\sum_{\gamma \in H}{\wh{1_A}(\gamma)^2} + O(\sup_{\gamma \neq 0_{\wh{G}}}{|\wh{1_A}(\gamma)|}\alpha)
\end{equation*} 
by Parseval's theorem. Note that if $\gamma \in H$ then $\gamma$ is a real character so $|\wh{1_A}(\gamma)|^2=\wh{1_A}(\gamma)^2$, thus we certainly have
\begin{equation*}
\sum_{\gamma \in H}{\wh{1_A}(\gamma)^2}=\sum_{\gamma \in H}{|\wh{1_A}(\gamma)|^2} \geq |\wh{1_A}(0_{\wh{G}})|^2=\alpha^2.
\end{equation*}
This is the previously mentioned positivity observation of Lev. It follows that either $\Lambda(A) \geq \alpha^3/2$ and we are done or $\sup_{\gamma \neq 0_{\wh{G}}}{|\wh{1_A}(\gamma)|} = \Omega(\alpha^2)$ in which case we apply Lemma \ref{lem.rmit} and are done.
\end{proof}
Now, the above lemma can be iterated to get Theorem \ref{thm.weak}, and again we shall use essentially the same style of iteration in \S\ref{sec.end} to prove Theorem \ref{thm.weighted}.
\begin{proof}[Proof of Theorem {\ref{thm.weak}} (sketch)]
We apply the preceding lemma repeatedly incrementing the density at each stage that we are in the second case of the lemma and terminating if we are in the first case.

At each stage we have $\alpha \mapsto \alpha+ \Omega(\alpha^2)$.  Thus, after $O(\alpha^{-1})$ iterations the density will have doubled. Since density cannot increase above $1$ we have that the iteration terminates after
\begin{equation*}
O(\alpha^{-1}) + O((2\alpha)^{-1}) + O((4\alpha^{-1})) + \dots = O(\alpha^{-1})
\end{equation*}
steps.

When the iteration terminates we have some group $G' \leq G$ with $|G:G'| = \exp(O(\alpha^{-1}))$ such that
\begin{equation*}
\Lambda(A)=\Omega(\alpha^3|G:G'|^2)=\exp(O(\alpha^{-1})).
\end{equation*}
The result follows.
\end{proof}

We shall exploit some of the additional structure of $\Z_4^n$ to effectively improve Lemma \ref{lem.itrm} and thereby gain our strengthening of the Roth-Meshulam-Lev argument.

In $G=\Z_4^n$ a triple $(x,y,z)$ with $x+z=2.y$ must have $x$ and $z$ in the same coset of $\Im 2$, where $2$ denotes the map $x \mapsto 2.x$. Thus it is natural to partition $A$ by the cosets of $\Im 2$, because when counting three-term arithmetic progressions we only ever need to consider sums $x+z$ with $x$ and $z$ in the same coset.

Since $\Im2= \ker 2$ we shall index the elements of this partition of $A$ by elements of $\ker 2$ and, for simplicity later, translate them all so that they lie in $\Im 2$.  Specifically then, we proceed as follows.

Suppose that $G$ is a finite abelian group and $A\subset G$. Define
\begin{equation*}
f_A:\Im2 \rightarrow [0,1]; u \mapsto \E_{z \in G:2.z=u}{1_A(z)},
\end{equation*}
and note that
\begin{eqnarray*}
\Lambda(A)& = & \E_{x,d \in G}{1_A(x)1_A(x+2.d)\E_{d' \in G: 2.d'=2.d}{1_A(x+d')}}\\ & = & \E_{x,d \in G}{1_A(x)1_A(x+2.d)f_A(2.(x+d))}\\ & = & \E_{x,u \in G}{1_A(x)1_A(2.u-x)f_A(2.u)}.
\end{eqnarray*}
Now, for each $y \in \Im 2$ let $t_y \in G$ be such that $2.t_y=y$, and let $A_y:=A \cap (t_y+\ker 2) -t_y \subset \ker 2$. Furthermore, for each $y \in \Im 2$ let $\tau_y$ be `translation by $y$' defined by
\begin{equation*}
\tau_y:L^1(\Im2) \rightarrow L^1(\Im2); f \mapsto (x \mapsto f(x+y)).
\end{equation*}
In this notation we have
\begin{eqnarray}
\label{eqn.ke}\Lambda(A)& = & \E_{y \in \Im2}{\E_{x \in t_y+\ker 2, v \in \Im2}{1_{A_y}(x-t_y)1_{A_y}(v-x-t_y)f_A(v)}}\\ \nonumber & = & \E_{y \in \Im2}{\E_{z \in \ker 2, v \in \Im2}{1_{A_y}(z)1_{A_y}(v-y-z)f_A(v)}}\\ \nonumber & = &\E_{y \in \Im2}{\langle \tau_y(1_{A_y} \ast 1_{A_y}), f_A\rangle_{L^2(\Im 2)}}.
\end{eqnarray}
Note that the convolution here denotes convolution on $\ker2$ since this is where the sets $A_y$ have been arranged to live.  In $\Z_4^n$ we have that $\Im2=\ker 2$ which simplifies this expression so that it only involves one group.

Our argument will consider two cases depending on whether or not $f_A$ supports large $L^2$-mass or not.
\begin{enumerate}
\item \emph{(Large $L^2$-mass)} Suppose that $\|f_A\|_{L^2(\Im 2)}^2 \geq \alpha^{5/3}$. Then, on average $A$ has density $\alpha^{2/3}$ on the fibres of the points in $2.A$. We wish to estimate inner products of the form $\langle \tau_y(1_{A_y} \ast 1_{A_y}),f_A\rangle_{L^2(\Im 2)}$ where the set $A_y$ is the fibre of $y$. Plancherel's theorem tells us that
\begin{eqnarray*}
\langle \tau_y(1_{A_y} \ast 1_{A_y}),f_A \rangle_{L^2(\Im 2)} & = & \sum_{\gamma \in \wh{G}}{|\wh{1_{A_y}}(\gamma)|^2\gamma(y)\wh{f_A}(\gamma)}\\ & = & \alpha_y^2\alpha + O(\sup_{\gamma \neq 0_{\wh{G}}}{|\wh{f_A}(\gamma)|}\alpha_y)
\end{eqnarray*}
where $\alpha_y$ is the density of the fibre. If $\alpha_y \geq \alpha^{2/3}$ then we get a non-trivial character at which $\wh{f_A}(\gamma) =\Omega(\alpha^{5/3})$. This leads to a corresponding density increment which could only be iterated $O(\alpha^{-2/3})$ times before the density would have to exceed 1.
\item \emph{(Small $L^2$-mass)} Suppose that $\|f_A\|_{L^2(\Im 2)}^2 \leq \alpha^{5/3}$. Then it follows that $2.A$ has density at least $\alpha^{1/3}$. We now replace $f_A$ with $1_{2.A}$ and find, in much the same way as above, that we have a non-trivial Fourier mode (this time of a fibre) of size $\Omega(\alpha^{5/3})$. If one could now perform a density increment in a way that was simultaneous for all fibres then this could only happen $O(\alpha^{-2/3})$ times.
\end{enumerate}
These two cases would combine to suggest that $A$ contained $\exp(-O(\alpha^{-2/3}))$ three-term arithmetic progressions. Unfortunately the second is too optimistic; the content of this paper is in making a version of the above sketch work and, in particular, dealing with the harder case of small $L^2$-mass.

\section{Families}\label{sec.families}

We make a new definition which we shall work with for the remainder of the paper to help simplify some of the inductive steps later and which should seem fairly natural given the discussion of the previous section.

Suppose that $H$ is a finite (abelian) group of exponent $2$. A \emph{family} on $H$ is a vector $\mathcal{A}=(A_h)_{h \in H}$ where $A_h \subset H$ for all $h \in H$; we call the set $A_h$ a \emph{fibre} of $\mathcal{A}$. We define the \emph{density function} of $\mathcal{A}$ to be
\begin{equation*}
f_\mathcal{A}:H \rightarrow [0,1]; h \mapsto \P_H(A_h),
\end{equation*}
and refer to $\E_{x \in H}{f_\mathcal{A}(x)}$ as the \emph{density} of $\mathcal{A}$ denoted $\P_H(\mathcal{A})$. 

We are interested in the quantity
\begin{equation*}
\Lambda(\mathcal{A}):=\E_{h \in H}{\langle \tau_h(1_{A_h} \ast 1_{A_h}),f_\mathcal{A} \rangle_{L^2(H)}},
\end{equation*}
and it is useful to note that $|H|^4\Lambda(\mathcal{A})$ is the number of quadruples $(a,a',y,h)$ with $a,a' \in A_h$ and $y \in A_{a+a'-h}$.

If $A \subset \Z_4^n$ then the family $\mathcal{A}:=(A_y)_{y \in \Im2}$ defined earlier for use in (\ref{eqn.ke}) has $\Lambda(\mathcal{A})=\Lambda(A)$ and density $\alpha$. Conversely, given any family $\mathcal{A}$ on $\Z_2^n$ we can clearly construct a set $A$ in $\Z_4^n$ such that $\Lambda(A)=\Lambda(\mathcal{A})$; families are simply a notational convenience. The bulk of the paper now concerns the proof that $\Lambda(\mathcal{A})$ is large in terms of the density of $\mathcal{A}$.

\section{Density increments on families}\label{sec.dense}

The arguments of this section are straightforward and encode the various ways in which we shall try to increment the density of our family under certain circumstances. The simplest of these is the standard $\ell^\infty$-density increment lemma which follows.
\begin{lemma}\label{lem.trivialdensityincrement}
Suppose that $H$ is a finite abelian group of exponent $2$, $f:H \rightarrow [0,1]$ and $\gamma$ is a non-trivial character. Then the subgroup $H':=\{\gamma\}^\perp$ has index $2$ and $\|f \ast \P_{H'}\|_{L^\infty(H)} = \E_{h \in H}{f(h)} + |\wh{f}(\gamma)|$.
\end{lemma}
\begin{proof}
Let $h_0 \in H \setminus H'$ so that $h_0+H'$ is the coset of $H'$ in $H$ not equal to $H'$. By definition
\begin{equation*}
|\wh{f}(\gamma)| = |\E_{h \in H}{1_{H'}(h)f(h)} - \E_{h \in H}{1_{h_0+H'}(h)f(h)}|.
\end{equation*}
We also have
\begin{equation*}
\E_{h \in H}{f(h)}=\E_{h \in H}{1_{H'}(h)f(h)} + \E_{h \in H}{1_{h_0+H'}(h)f(h)},
\end{equation*}
which on being added to the previous tells us that
\begin{equation*}
2.\max\{\E_{h \in H}{1_{H'}(h)f(h)},\E_{h \in H}{1_{h_0+H'}(h)f(h)}\} = \E_{h \in H}{f(h)}+|\wh{f}(\gamma)|.
\end{equation*}
Since the index of $H'$ in $H$ is $2$ we have that $2(\P_H)|_{H'}=\P_{H'}$, whence
\begin{equation*}
\max\{(f \ast \P_{H'})(0_H), (f \ast \P_{H'})(h_0)\} = \E_{h \in H}{f(h)}+|\wh{f}(\gamma)|,
\end{equation*}
and the result follows.
\end{proof}
The next lemma is a sort of simultaneous version of the above. If a family has a large number of its fibres having a large Fourier coefficient at the same non-trivial character $\gamma$ then there is a related family with increased density.
\begin{lemma}\label{lem.coreincrement}
Suppose that $H$ is a finite abelian group of exponent $2$, $\mathcal{A}=(A_h)_{h \in H}$ is a family on $H$ and $\gamma$ is a non-trivial character. Then there is a subgroup $H'\leq H$ of index $2$  and a family $\mathcal{A}'$ on $H'$ such that
\begin{equation*}
\Lambda(\mathcal{A}) \geq 2^{-4}\Lambda(\mathcal{A}') \textrm{ and } \P_{H'}(\mathcal{A}') \geq \P_H(\mathcal{A}) + \E_{h \in H}{|\wh{1_{A_h}}(\gamma)|}.
\end{equation*}
\end{lemma}
\begin{proof}
Let $H':=\{\gamma\}^\perp$ and let $h_0 \in H \setminus H'$ so that $h_0+H'$ is the coset of $H'$ in $H$ not equal to $H'$. For each $h \in H$ apply Lemma \ref{lem.trivialdensityincrement} to see that
\begin{equation*}
\|1_{A_h} \ast \P_{H'}\|_{L^\infty(H)} \geq \E_{\tilde{h} \in H}{1_{A_h}(\tilde{h})} + |\wh{1_{A_h}}(\gamma)|.
\end{equation*}
Now let $x_h \in H$ be such that $\|1_{A_h} \ast \P_{H'}\|_{L^\infty(H)} =(1_{A_h}\ast \P_{H'})(x_h)$ and define $B_h:=A_h \cap (x_h+H') - x_h \subset H'$, whence
\begin{equation*}
\P_{H'}(B_h) \geq \E_{\tilde{h} \in H}{1_{A_h}(\tilde{h})} + |\wh{1_{A_h}}(\gamma)|=f_\mathcal{A}(h)+|\wh{1_{A_h}}(\gamma)|.
\end{equation*}
It follows that
\begin{equation*}
\E_{h \in H}{\P_{H'}(B_h) } \geq \E_{h \in H}{f_\mathcal{A}(h)}+\E_{h \in H}{|\wh{1_{A_h}}(\gamma)|},
\end{equation*}
whence by averaging there is a coset $h_1+H'$ of $H$ such that
\begin{equation*}
\E_{h \in h_1+H'}{\P_{H'}(B_h) } \geq \E_{h \in H}{f_\mathcal{A}(h)}+\E_{h \in H}{|\wh{1_{A_h}}(\gamma)|}.
\end{equation*}
Now we define a family $\mathcal{A}'$ on $H'$ as follows: for each $h' \in H'$ let $A_{h'}':=B_{h_1+h'}$. It is immediate that $\mathcal{A}'$ has the required density; it remains to show that $\Lambda(\mathcal{A}) \geq 2^{-4}\Lambda(\mathcal{A}')$, which is a relatively simple counting exercise.

There are $|H'|^4\Lambda(\mathcal{A}')$ quadruples $(a_0',a_1',y',h')$ with $a_0', a_1' \in A_{h'}'$ and $y' \in A_{a_0'+a_1'-h'}'$. Every such quadruple corresponds uniquely to a quadruple
\begin{equation*}
(a_0,a_1,y,h):=(a_0'+x_{h_1+h'},a_1'+x_{h_1+h'},y'+x_{a_0'+a_1'-h'+h_1},h_1+h');
\end{equation*}
unique since there is an obvious inverse on the image taking $(a_0,a_1,y,h)$ to
\begin{equation*}
(a_0',a_1',y',h')=(a_0-x_h,a_1-x_h,y-x_{a_0+a_1-h-2.x_h+2.h_1},h-h_1).
\end{equation*}
Now,
\begin{equation*}
a_0 = a_0'+x_{h_1+h'} \in A_h' +x_{h_1+h'} = B_{h_1+h'} + x_{h_1+h'} \subset A_{h_1+h'}=A_h,
\end{equation*}
and similarly $a_1 \in A_h$. Furthermore
\begin{eqnarray*}
y=y'+x_{a_0'+a_1'-h'+h_1} &\in &A_{a_0'+a_1'-h'}' + x_{a_0'+a_1'-h'+h_1} \\& = & B_{a_0'+a_1'-h'+h_1}+x_{a_0'+a_1'-h'+h_1}\\ & \subset & A_{a_0'+a_1'-h'+h_1} = A_{a_0+a_1-h}
\end{eqnarray*}
since $2.x_{h_1+h'}=0_H$ and $2.h_1=0_H$. It follows that every quadruple $(a_0',a_1',y',h')$ with $a_0',a_1' \in A_{h'}'$ and $y' \in A_{a_0'+a_1'-h'}'$ corresponds to a unique quadruple $(a_0,a_1,y,h)$ with $a_0,a_1 \in A_h$ and $y \in A_{a_0+a_1-h}$, whence $|H'|^4\Lambda(\mathcal{A}) \leq |H|^4\Lambda(\mathcal{A})$ and the result follows on noting that $|H|^4 = 2^4 |H'|^4$.
\end{proof}
The last part of the above proof is a rather fiddly verification of a type which we shall have to do repeatedly, and while we have been rather comprehensive in the details above, in the future we shall include fewer of them.

The final lemma of the section takes a family where the density function is non-uniform and produces a new family with a larger density, again very much in the spirit of the previous to lemmas.
\begin{lemma}\label{lem.coreincrement2}
Suppose that $H$ is a finite abelian group of exponent $2$, $\mathcal{A}=(A_h)_{h \in H}$ is a family on $H$ and $\gamma$ is a non-trivial character. Then there is a subgroup $H'\leq H$ of index $2$  and a family $\mathcal{A}'$ on $H'$ such that
\begin{equation*}
\Lambda(\mathcal{A}) \geq 2^{-4}\Lambda(\mathcal{A}') \textrm{ and } \P_{H'}(\mathcal{A}') \geq \P_H(\mathcal{A}) + |\wh{f_\mathcal{A}}(\gamma)|.
\end{equation*}
\end{lemma}
\begin{proof}
Let $H':=\{\gamma\}^\perp$ and let $h_0 \in H \setminus H'$ so that $h_0+H'$ is the coset of $H'$ in $H$ not equal to $H'$. Apply Lemma \ref{lem.trivialdensityincrement} so that we have
\begin{equation*}
\|f_\mathcal{A} \ast \P_{H'}\|_{L^\infty(H)} = \P_H(\mathcal{A}) + |\wh{f_\mathcal{A}}(\gamma)|.
\end{equation*}
Let $h_1 \in H'$ be such that $(f_\mathcal{A}\ast\P_{H'})(h_1)=\|f_\mathcal{A} \ast \P_{H'}\|_{L^\infty(H)}$. Now, define a family $\mathcal{A}'$ as follows: for each $h' \in H'$
\begin{enumerate}
\item if $A_{h_1+h'} \cap H'$ is larger than $A_{h_1+h'} \cap (h_0+H')$ then put $x_{h_1+h'}:=0_H$ and $A_{h'}':=A_{h_1+h'} \cap H'$;
\item otherwise put $x_{h_1+h'}:=h_0$ and $A_{h'}':=A_{h_1+h'} \cap (h_0+H')-h_0$.
\end{enumerate}
By averaging we have that $\P_{H'}(A_{h'}') \geq \P_H(A_{h_1+h'})$, whence
\begin{equation*}
\E_{h' \in H'}{\P_{H'}(A_{h'}')} \geq (f_\mathcal{A}\ast\P_{H'})(h_1) = \P_H(\mathcal{A}) + |\wh{f_\mathcal{A}}(\gamma)|,
\end{equation*}
which yields the required density condition. It remains, as before, to show that $\Lambda(\mathcal{A}) \geq 2^{-4}\Lambda(\mathcal{A}')$; we proceed as in the previous lemma.

There are $|H'|^4\Lambda(\mathcal{A}')$ quadruples $(a_0',a_1',y',h')$ with $a_0', a_1' \in A_{h'}'$ and $y' \in A_{a_0'+a_1'-h'}'$. Every such quadruple corresponds uniquely to a quadruple
\begin{equation*}
(a_0,a_1,y,h):=(a_0'+x_{h_1+h'},a_1'+x_{h_1+h'},y'+x_{a_0'+a_1'-h'+h_1},h_1+h')
\end{equation*}
with $a_0,a_1 \in A_h$ and $y \in A_{a_0+a_1-h}$, whence $|H'|^4\Lambda(\mathcal{A}) \leq |H|^4\Lambda(\mathcal{A})$ and the result follows on noting that $|H|^4 = 2^4 |H'|^4$.
\end{proof}

\section{Families with large mean square density}

In this section we show how a family $\mathcal{A}$ for which $\|f_\mathcal{A}\|_{L^2(H)}$ is large (compared with its trivial lower bound of $\P_H(\mathcal{A})^2$) has $\Lambda(\mathcal{A})$ large. The basic idea is that if $\|f_\mathcal{A}\|_{L^2(H)}$ is large then most of the fibres $A_h$ have large density and so are more easily `uniformized'. When they are uniform the count $\Lambda(\mathcal{A})$ is easily seen to be large.

It is instructive to consider a simplified situation. Suppose that $\mathcal{A}$ is a family which is assumed to have fibres of density either $0$ or $\delta$ and the support of $f_\mathcal{A}$ has density $\sigma$. This family has density $\delta \sigma$ and $\|f_\mathcal{A}\|_{L^2(H)}^2=\delta^2\sigma$, which is large compared with the trivial lower bound of $(\delta\sigma)^2$ if $\sigma$ is small. Now, the standard Roth-Meshulam argument can be used to show that $\Lambda(\mathcal{A}) = \exp(O(\delta^{-1}\sigma^{-1}))$, and the proposition below asserts that this can be improved when $\sigma$ is small.
\begin{proposition}\label{prop.largeL}
Suppose that $H$ is a finite abelian group of exponent $2$ and $\mathcal{A}=(A_h)_{h \in H}$ is a family on $H$ such that $f_\mathcal{A}=\delta1_S$ for some $\delta \in (0,1]$ and $S \subset H$ of density $\sigma$. Then $\Lambda(\mathcal{A}) = \exp(-O(\delta^{-1}\log \sigma^{-1}))$
\end{proposition}
Naturally the proof is iterative with the following lemma acting as the driver.
\begin{lemma}\label{lem.itlem1}
Suppose that $H$ is a finite abelian group of exponent $2$, $\mathcal{A}=(A_h)_{h \in H}$ is a family on $H$ and $f_\mathcal{A} = \delta1_S$ for some $\delta \in (0,1]$ and $S \subset H$ of density $\sigma$. Then either
\begin{equation*}
\Lambda(\mathcal{A}) \geq \delta^3\sigma^2/2,
\end{equation*}
or there is a subgroup $H' \leq H$ of index $2$, a family $\mathcal{A}'$ on $H'$ and set $S' \subset H'$ such that $f_{\mathcal{A}'}= \delta 1_{S'}$ and
\begin{equation*}
\P_{H'}(S') \geq \sigma(1+\delta/2) \textrm{ and } \Lambda(\mathcal{A}) \geq 2^{-4}\Lambda(\mathcal{A}').
\end{equation*}
\end{lemma}
\begin{proof}
Since $f_\mathcal{A}=\delta 1_S$ we have that
\begin{equation*}
\Lambda(\mathcal{A}) = \delta \E_{h \in H}{\langle \tau_h(1_{A_h} \ast 1_{A_h}),1_S \rangle_{L^2(H)}}.
\end{equation*}
Applying Plancherel's theorem to the inner products we get that
\begin{equation*}
\Lambda(\mathcal{A}) = \delta \E_{h \in H}{\sum_{\gamma \in \wh{H}}{|\wh{1_{A_h}}(\gamma)|^2\wh{1_S}(\gamma)\gamma(h)}}.
\end{equation*}
The triangle inequality may be used on these inner sums to separate out the trivial mode. Indeed, since $\wh{1_{A_h}}(0_{\wh{H}})=f_\mathcal{A}(h)$ and $\wh{1_S}(0_{\wh{H}})=\sigma$ we get -- after a little manipulation -- that
\begin{eqnarray*}
\E_{h \in H}{\sum_{\gamma \neq 0_{\wh{H}}}{|\wh{1_{A_h}}(\gamma)|^2|\wh{1_S}(\gamma)|}} & \geq & \E_{h \in H}{f_\mathcal{A}(h)^2}\sigma - \delta^{-1}\Lambda(\mathcal{A})\\ & = & \delta^2\sigma^2 - \delta^{-1}\Lambda(\mathcal{A}).
\end{eqnarray*}
Now, we are done unless $\Lambda(\mathcal{A}) \leq \delta^3\sigma^2/2$ (in fact, unless $\Lambda(\mathcal{A}) < \delta^3\sigma^2/2$ but we shall not use this), whence
\begin{equation*}
\E_{h \in H}{\sum_{\gamma \neq 0_{\wh{H}}}{|\wh{1_{A_h}}(\gamma)|^2|\wh{1_S}(\gamma)|}}\geq  \delta^2\sigma^2/2.
\end{equation*}
On the other hand
\begin{equation*}
\E_{h \in H}{\sum_{\gamma \neq 0_{\wh{H}}}{|\wh{1_{A_h}}(\gamma)|^2}} = \E_{h \in H}{(f_\mathcal{A}(h) - f_\mathcal{A}(h)^2)} =\delta(1-\delta)\sigma \leq \delta \sigma
\end{equation*}
by Parseval's theorem. Using this with the triangle inequality in the previous expression tells us that $S$ is linearly biased:
\begin{equation*}
\sup_{\gamma \neq 0_{\wh{H}}}{|\wh{1_S}(\gamma)|} \geq \delta\sigma/2.
\end{equation*}
Thus, by Lemma \ref{lem.trivialdensityincrement} there is a subgroup $H' \leq H$ of index $2$ such that
\begin{equation}\label{eqn.increment}
\|1_S \ast \P_{H'}\|_{L^\infty(H)} \geq \sigma(1+\delta/2).
\end{equation}
Let $h_1 \in H$ be such that $(1_S \ast \P_{H'})(h_1) = \|1_S \ast \P_{H'}\|_{L^\infty(H)}$ and define a family $\mathcal{A}':=(A'_{h'})_{h' \in H'}$ as follows. For each $h' \in H'$ let $x_{h'+h_1}$ be such that $(1_{A_{h'+h_1}} \ast \P_{H'})(x_{h'+h_1})$ is maximal. If $(1_{A_{h'+h_1}} \ast \P_{H'})(x_{h'+h_1})>0$ then 
\begin{equation*}
0 < (1_{A_{h'+h_1}} \ast \P_{H'})(x_{h'+h_1})/2 \leq f_{\mathcal{A}}(h'+h_1) =\delta 1_S(h'+h_1),
\end{equation*}
whence
\begin{equation*}
1_{A_{h'+h_1}} \ast \P_{H'}(x_{h'+h_1}) \geq \E_{h \in H}{1_{A_{h'+h_1}}(h)}=f_\mathcal{A}(h'+h_1)=\delta,\end{equation*}
and $A_{h'+h_1} \cap (x_{h'+h_1}+H') - x_{h'+h_1}$ contains a set of density $\delta$; let $A_{h'}'$ be such a set. If $(1_{A_{h'+h_1}} \ast \P_{H'})(x_{h'+h_1})=0$ then let $A_{h'}'=\emptyset$. Finally, we write $S':=S \cap (h_1+H')-h_1$ and it remains to check that we have the required properties.

First, note that 
\begin{equation*}
f_{\mathcal{A}'}(h') =\P_{H'}(A_{h'}') \leq 2\P_{H}(A_{h'+h_1})=2f_{\mathcal{A}}(h'+h_1),
\end{equation*}
thus if $f_{\mathcal{A}'}(h')>0$ then $h'+h_1 \in S$ and so $h' \in S'$. Similarly,
\begin{equation*}
f_{\mathcal{A}'}(h') =\P_{H'}(A_{h'}') \geq \P_{H}(A_{h'+h_1})=2f_{\mathcal{A}}(h'+h_1),
\end{equation*}
so if $f_{\mathcal{A}'}(h')=0$ then $h'+h_1 \not \in S$, whence $h' \not \in S'$. By design, $f_{\mathcal{A}'}$ takes only the values $0$ and $\delta$ and so we have the representation $f_{\mathcal{A}'}=\delta 1_{S'}$.

Secondly, we have that $\P_{H'}(S') = 1_S \ast \P_{H'}(h_1)$, whence $\P_{H'}(S') \geq \sigma(1+\delta/2)$ by (\ref{eqn.increment}). Lastly, we check that  $\Lambda(\mathcal{A}) \geq 2^{-4}\Lambda(\mathcal{A}')$ in the usual fashion.

There are $|H'|^4\Lambda(\mathcal{A}')$ quadruples $(a_0',a_1',y',h')$ with $a_0', a_1' \in A_{h'}'$ and $y' \in A_{a_0'+a_1'-h'}'$. Every such quadruple corresponds uniquely to a quadruple
\begin{equation*}
(a_0,a_1,y,h):=(a_0'+x_{h_1+h'},a_1'+x_{h_1+h'},y'+x_{a_0'+a_1'-h'+h_1},h_1+h')
\end{equation*}
with $a_0,a_1 \in A_h$ and $y \in A_{a_0+a_1-h}$, whence $|H'|^4\Lambda(\mathcal{A}) \leq |H|^4\Lambda(\mathcal{A})$ and the result follows on noting that $|H|^4 = 2^4 |H'|^4$.
\end{proof}

\begin{proof}[Proof of Proposition \ref{prop.largeL}]
Let $H_0:=H$, $\mathcal{A}_0:=\mathcal{A}$, $\alpha_0:=\delta \sigma$, $S_0:=S$ and $\sigma_0:=\sigma$. Suppose that we have a finite abelian group $H_i$ of exponent $2$ with a family $\mathcal{A}_i$ on $H_i$ of density $\alpha_i$ and a set $S_i$ of density $\sigma_i$ such that $f_{\mathcal{A}_i}=\delta1_{S_i}$. Apply Lemma \ref{lem.itlem1} to see that either
\begin{equation*}
\Lambda(\mathcal{A}_i) \geq  \delta^3\sigma_i^2/2,
\end{equation*}
or there is a subgroup $H_{i+1}$ of index $2$ in $H_i$, a family $\mathcal{A}_{i+1}$ and a set $S_{i+1}$ such that
\begin{equation*}
f_{\mathcal{A}_{i+1}} = \delta1_{S_{i+1}}, \sigma_{i+1} \geq \sigma_i(1+ \delta/2) \textrm{ and } \Lambda(\mathcal{A}_i) \geq 2^{-4} \Lambda(\mathcal{A}_{i+1}).
\end{equation*}
Since $\sigma_i \leq 1$ we see that this iteration must terminate at some stage $i$ with $(1+\delta/2)^i \leq \sigma^{-1}$ i.e. with $i \leq 2\delta^{-1}\log \sigma^{-1}$. It follows that
\begin{equation*}
\Lambda(\mathcal{A}) \geq 2^{-8\delta^{-1}\log \sigma^{-1}}\delta^3\sigma^2/2,
\end{equation*}
which is the result
\end{proof}
Proposition \ref{prop.largeL} will be used again as is in \S\ref{sec.highfibre} but it may seem like the rather special form of the family considered is too restrictive. However, a standard dyadic decomposition lets us apply this proposition to an arbitrary family; we gain precisely in the case when $\|f_\mathcal{A}\|_{L^2(H)}^2\alpha^{-2} \rightarrow \infty$.
\begin{corollary}\label{cor.kuy}
Suppose that $H$ is a finite abelian group of exponent $2$, $\mathcal{A}=(A_h)_{h \in H}$ is a family on $H$ of density $\alpha$ and $\|f_\mathcal{A}\|_{L^2(H)}=K\alpha^2$ for some $K \geq 2$. Then\begin{equation*}
\Lambda(\mathcal{A}) = \exp(-O(\alpha^{-1} K^{-1}\log^2 K)).
\end{equation*}
\end{corollary}
\begin{proof}
Let $S_i:=\{h \in H: 2^{-(i+1)} \leq f_\mathcal{A}(h) \leq 2^{-i}\}$ and $S':=\{h \in H: f_\mathcal{A}(h) \leq \alpha/2\}$. We may use these sets to partition the range of summation in $\|f_\mathcal{A}\|_{L^2(H)}^2$: by the triangle inequality
\begin{equation*}
\sum_{i \leq \lceil \log_2 \alpha^{-1} \rceil}{2^{-2i}\P_H(S_i)} + (\alpha/2)^2 \geq \|f_\mathcal{A}\|_{L^2(H)}^2.
\end{equation*}
The Cauchy-Schwarz inequality tells us that $\|f_{\mathcal{A}}\|_{L^2(H)}^2 \geq \alpha^2$, whence 
\begin{equation}\label{eqn.lk}
\sum_{i \leq \lceil \log_2 \alpha^{-1} \rceil}{2^{-2i}\P_H(S_i)}  \geq 3\|f_\mathcal{A}\|_{L^2(H)}^2/4.
\end{equation}
Now let $\epsilon \in (0,1]$ be a parameter to be chosen later and note that
\begin{eqnarray*}
\sum_{i \leq \lceil \log_2 \alpha^{-1} \rceil}{2^{\epsilon i}} & \leq & 2\alpha^{-\epsilon} \sum_{ i \leq \lceil \log_2 \alpha^{-1} \rceil}{2^{\epsilon(i-\lceil \log_2 \alpha^{-1} \rceil)}}\\ & \leq & 2\alpha^{-\epsilon} \sum_{j=0}^\infty{2^{-\epsilon j}}\\ & = & 2\alpha^\epsilon/(1-2^{-\epsilon}) \leq 2\epsilon^{-1}\alpha^{-\epsilon}.
\end{eqnarray*}
Returning to (\ref{eqn.lk}) we see that
\begin{equation*}
\sum_{i \leq \lceil \log_2 \alpha^{-1} \rceil}{2^{\epsilon i}.2^{-(2+\epsilon)i}\P_H(S_i)}  \geq 3\|f_\mathcal{A}\|_{L^2(H)}^2/4,
\end{equation*}
and so by averaging from our previous calculation there is some $i \leq \lceil \log_2 \alpha^{-1} \rceil$ such that
\begin{equation}\label{eqn.j}
(2\epsilon^{-1}\alpha^{-\epsilon}). 2^{-(2+\epsilon)i}\P_H(S_i) \geq 3\|f_\mathcal{A}\|_{L^2(H)}^2/4.
\end{equation}
Moreover $2^{-(i+1)}\P_H(S_i) \leq \E_{h \in H}{f_\mathcal{A}(h)}=\alpha$; so, recalling that $\|f_\mathcal{A}\|_{L^2(H)}^2=K\alpha^2$ we have
\begin{equation*}
2^{-(1+\epsilon)i} \geq 3 \epsilon K\alpha^{1+\epsilon}/16.
\end{equation*}
If we take $\epsilon = 1/(1+\log K)$ then we get that
\begin{equation}\label{eqn.dyadic}
2^{-(i+1)} = \Omega(\alpha K/\log K).
\end{equation}
Let $\mathcal{A}'$ be a family defined as follows. If $h \in S_i$ then $A'_{h}$ is a subset of $A_h$ of density $2^{-(i+1)}$ and $A'_h$ is empty otherwise. By comparison of the terms in $\Lambda(\mathcal{A})$ with those in $\Lambda(\mathcal{A}')$ we see that $\Lambda(\mathcal{A}) \geq \Lambda(\mathcal{A}')$.

We now apply Proposition \ref{prop.largeL} to $\mathcal{A}'$; it is easy to see from (\ref{eqn.j}) and (\ref{eqn.dyadic}) that 
\begin{equation*}
\delta^{-1}=2^{(i+1)} = O( \alpha^{-1}K^{-1}\log K)
\end{equation*}
and
\begin{equation*}
\log \sigma^{-1} =\log \P_H(S_i)^{-1} = O(\log((\delta \alpha^{-1})^{2+\epsilon}K^{-1}\log K)) = O(\log K\delta \alpha^{-1}).
\end{equation*}
The result follows on noting that
\begin{equation*}
\Lambda(\mathcal{A}) = \exp(-O(\delta^{-1}\log K\delta\alpha^{-1}))
\end{equation*}
increases as $\delta$ decreases.
\end{proof}
It should be noted that one cannot completely remove the logarithmic term in this corollary. We might have $K \sim \alpha^{-1}$, but $\Lambda(\mathcal{A})$ may still be $\exp(-\Omega(\log K))$. To see this consider, for example, the family $\mathcal{A}$ where every fibre $A_h$ is a random set of density $\alpha$. Of course, the logarithmic power will not significantly affect our final result and is only critical when $K$ is much smaller than $\alpha^{-1}$, in which case it may be possible to remove it entirely.

\section{A quasi-random Balog-Szemer{\'e}di-Gowers-Fre{\u\i}man theorem}\label{sec.d}

The Balog-Szemer{\'e}di-Gowers-Fre{\u\i}man theorem is a now ubiquitous result in additive combinatorics   introduced by Gowers in \cite{WTG}. It combines (a refined proof of) the Balog-Szemer{\'e}di theorem \cite{ABES} with the structure theorem of Fre{\u\i}man \cite{GAF} concerning sets with small sumset. Since we are working in finite abelian groups of exponent $2$ we actually require the far easier torsion version of Fre{\u\i}man's theorem due to Ruzsa \cite{IZRArb}. In fact, in this setting a version of the Balog-Szemer{\'e}di-Gowers-Fre{\u\i}man theorem is known with relatively good bounds.
\begin{theorem}[{\cite[Theorem 1.7]{BJGTCTF}}]\label{thm.bsg}
Suppose that $H$ is a group of exponent $2$, $A \subset H$ has density $\alpha$ and $\|1_A \ast 1_A \|_{L^2(H)}^2 \geq c\alpha^3$. Then there is an element $x\in H$ and a subgroup $H' \leq H$ such that
\begin{equation*}
\P_H(H') = \exp(-O(c^{-1} \log c^{-1}))\alpha \textrm{ and } (1_A \ast \P_{H'})(x) \geq c/2.
\end{equation*}
\end{theorem}
We actually require a slightly modified version of this result which also ensures that $A'$ behaves uniformly on $H'$. This can essentially be read out of the proof of Green and Tao \cite{BJGTCTF}; however, for completeness, we include a `de-coupled' proof here.
\begin{corollary}\label{cor.bsgcor}
Suppose that $H$ is a group of exponent $2$, $A \subset H$ has density $\alpha$ and $\|1_A \ast 1_A \|_{L^2(H)}^2 \geq c\alpha^3$, and $\epsilon \in (0,1]$ is a parameter. Then there is an element $x \in H$ and a subgroup $H' \leq H$ such that
\begin{equation*}
\P_H(H') = \exp(-O((c^{-1} +\epsilon^{-1})\log c^{-1}))\alpha \textrm{ and } (1_A \ast \P_{H'})(x) \geq c/2,
\end{equation*}
and writing $A':=A \cap (x+H') -x \subset H'$ one has
\begin{equation*}
\sup_{\gamma \neq 0_{\wh{H'}}}{|\wh{1_{A'}}(\gamma)|} \leq \epsilon \P_{H'}(A').
\end{equation*}
\end{corollary}
\begin{proof}
We apply the previous theorem (Theorem \ref{thm.bsg}) to get an element $x_0 \in H$ and a subgroup $H_0 \leq H$ such that
\begin{equation*}
\P_H(H_0) = \exp(-O(c^{-1} \log c^{-1})) \textrm{ and } (1_A \ast \P_{H_0})(x_0) \geq c/2.
\end{equation*}
Put $A_0:=A \cap (x_0+H_0)-x_0 \subset H_0$ and $\alpha_0:=\P_{H_0}(A_0)$.  Now, suppose that we have been given an element $x_i \in H$, a subgroup $H_i$ and a subset $A_i$ of $H_i$ of density $\alpha_i$. If
\begin{equation}\label{eqn.r}
\sup_{\gamma \neq 0_{\wh{H_i}}}{|\wh{1_{A_i}}(\gamma)|} \leq \epsilon \alpha_i
\end{equation}
then we terminate the iteration; otherwise we apply Lemma \ref{lem.trivialdensityincrement} to get a subspace $H_{i+1}$ of index $2$ in $H_i$ such that
\begin{equation*}
\|1_{A_i} \ast \P_{H_{i+1}}\|_{L^\infty(H_i)} \geq \alpha_i(1+\epsilon).
\end{equation*}
Let $x_{i+1}$ be such that $(1_{A_i} \ast \P_{H_{i+1}})(x_{i+1})=\|1_{A_i} \ast \P_{H_{i+1}}\|_{L^\infty(H_i)}$, and $A_{i+1} = A_i \cap (x_{i+1} + H_{i+1}) - x_{i+1} \subset H_{i+1}$. 

Since $\alpha_i \leq 1$ we see that this iteration must terminate at some stage $i$ with $(1+\epsilon)^i \leq \alpha_0^{-1}$ i.e. with $i \leq \epsilon^{-1}\log \alpha_0^{-1}=O(\epsilon^{-1}\log c^{-1})$. We put $x:=x_0+\dots+x_i$ and $H':=H_i$ so that $H_i$ has index $O(\epsilon^{-1}\log c^{-1})$ in $H_0$ and $A'=A_i$ has density at least $c/2$. Thus
\begin{equation*}
\P_H(H') =\P_H(H_0).\P_{H_0}(H_k) =\exp(-O((c^{-1} +\epsilon^{-1})\log c^{-1}))\alpha,
\end{equation*}
and it remains to note that the final condition of the corollary holds in view of the fact that we must have (\ref{eqn.r}) for the iteration to terminate.
\end{proof}
The iteration in the above proof is essentially the iteration at the core of the usual Roth-Meshulam argument (c.f. \S\ref{sec.out}) and consequently if one could improve the $\epsilon$-dependence in the above result one could probably improve the Roth-Meshulam argument directly. Unfortunately in our use of this corollary $\epsilon$ and $c$ are comparable; thus, even in the presence of Marton's conjecture, more commonly called the Polynomial Fre{\u\i}man Ruzsa conjecture (see \cite{BJGFFM}), we would see no significant improvement in our final result.

\section{Families with high fibered energy}\label{sec.highfibre}

In this section we use our previous work to show that if a family $\mathcal{A}$ has large additive energy in its fibres then $\Lambda(\mathcal{A})$ is large. The actual statement of the result is rather technical so we take a moment now to sketch the approach.

The key tool is the corollary of the Balog-Szemer{\'e}di-Gowers-Fre{\u\i}man theorem established in \S\ref{sec.d}. This may be applied individually to the fibres of $\mathcal{A}$ in each case, producing a subgroup on which the fibre is very dense. If all of these subgroups are very different then it is easy to see that $\Lambda(\mathcal{A})$ must be large; if not then by expanding them a little bit we find one subgroup on which a lot of fibres of $\mathcal{A}$ are very dense and we may use Proposition \ref{prop.largeL} to get that $\Lambda(\mathcal{A})$ is large.

Concretely, then, the purpose of this section is to prove the following.
\begin{lemma}\label{prop.highnrg}
Suppose that $H$ is a finite abelian group of exponent $2$, and $\mathcal{A}=(A_h)_{h \in H}$ is a family on $H$ of density $\alpha$ such that
\begin{equation*}
\sup_{\gamma \neq 0_{\wh{H}}}{|\wh{f_\mathcal{A}}(\gamma)|} \leq L \alpha^2,
\end{equation*}
for some parameter $L \geq 1$. Suppose, further that $S$ is a set of density $\sigma$ and $K\geq 1$ is a parameter such that
\begin{enumerate}
\item $K\alpha \geq f_\mathcal{A}(h) \geq K \alpha/2$ for all $h \in S$;
\item and $\|1_{A_h} \ast 1_{A_h}\|_{L^2(H)}^2 \geq cf_{\mathcal{A}}(h)^3$ for all $h \in S$.
\end{enumerate}
Then
\begin{equation*}
\Lambda(\mathcal{A}) \geq  \exp(-O(L(\log^2\alpha^{-1}+\log \sigma^{-1})\exp(O((c^{-1}+K^{1/2}c^{-1/2})\log c^{-1})))).
\end{equation*}
\end{lemma}
\begin{proof}
By Corollary \ref{cor.bsgcor} (with $\epsilon = 2^{-2}\sqrt{c/K}$) we see that for each $h \in S$ there is an element $x_h \in H$ and a subgroup $H_h \leq H$ such that the set $A_h':=A_h \cap (x_h+H_h)-x_h$ has
\begin{equation*}
\P_H(H_h) =\exp(-O((c^{-1} +K^{1/2}c^{-1/2})\log c^{-1}))f_\mathcal{A}(h) \textrm{ and } \P_{H_h}(A_h')\geq c/2,
\end{equation*}
and, furthermore,
\begin{equation}\label{eqn.qrness}
\sup_{\gamma \neq 0_{\wh{H_h}}}{|\wh{1_{A_h'}}(\gamma)|} \leq \epsilon \P_{H_h}(A_h').
\end{equation}

Now, let $S_0:=\{h \in S: (f_\mathcal{A}\ast \P_{H_h})(h) \geq \alpha/2\}$ and $S_1 :=S\setminus S_0$; we shall now split into two cases according to which of $S_0$ or $S_1$ is larger.

\begin{case} Suppose that $\P_H(S_0) \geq \sigma/2$. Then
\begin{equation*}
\Lambda(\mathcal{A}) \geq \alpha^3\sigma \exp(-O((c^{-1}+K^{1/2}c^{-1/2})\log c^{-1})).
\end{equation*}
\end{case}
\begin{proof}
By non-negativity of the terms in $\Lambda(\mathcal{A})$ we have that
\begin{equation*}
\Lambda(\mathcal{A})  \geq \E_{h \in H}{1_{{S_0}}(h)\langle \tau_h(1_{A_h} \ast 1_{A_h}),f_\mathcal{A} \rangle_{L^2(H)}}.
\end{equation*}
We analyse these inner products individually. Suppose that $h \in S_0$ and note that
\begin{equation*}
\langle \tau_h(1_{A_h} \ast 1_{A_h}),f_\mathcal{A} \rangle_{L^2(H)} \geq \P_H(H_h)^2\langle \tau_h(1_{A_h'} \ast 1_{A_h'}),f_\mathcal{A} \rangle_{L^2(h+H_h)}.
\end{equation*}
As usual this inner product is analysed using the Fourier transform: by Plancherel's theorem we have that
\begin{equation*}
\langle 1_{A_h'} \ast 1_{A_h'},\tau_{-h}(f_\mathcal{A}) \rangle_{L^2(H_h)} = \sum_{\gamma \in \wh{H_h}}{|\wh{1_{A_h'}}(\gamma)|^2\wh{\tau_{-h}(f_\mathcal{A})}(\gamma)}.
\end{equation*}
Separating out the contribution from the trivial character we get
\begin{eqnarray}
\label{eqn.earlier} \langle 1_{A_h'} \ast 1_{A_h'},\tau_{-h}(f_\mathcal{A}) \rangle_{L^2(H_h)} &  \geq & \P_{H_h}(A_h')^2(f_\mathcal{A}\ast \P_{H_h})(h)\\ \nonumber & & -  \sum_{\gamma \neq 0_{\wh{H_h}}}{|\wh{1_{A_h'}}(\gamma)|^2|\wh{\tau_{-h}(f_\mathcal{A})}(\gamma)|}.
\end{eqnarray}
This last term sum can be estimated as follows using H{\"o}lder's inequality and the Cauchy-Schwarz inequality:
\begin{eqnarray*}
\sum_{\gamma \neq 0_{\wh{H_h}}}{|\wh{1_{A_h'}}(\gamma)|^2|\wh{\tau_{-h}(f_\mathcal{A})}(\gamma)|}  & \leq & \sup_{\gamma \neq 0_{\wh{H_h}}}{|\wh{1_{A_h'}}(\gamma)|}.\sum_{\gamma \in \wh{H_h}}{|\wh{1_{A_h'}}(\gamma)||\wh{\tau_{-h}(f_\mathcal{A})}(\gamma)|} \\ & \leq &\sup_{\gamma \neq 0_{\wh{H_h}}}{|\wh{1_{A_h'}}(\gamma)|}.\left(\sum_{\gamma \in \wh{H_h}}{|\wh{1_{A_h'}}(\gamma)|^2}\right)^{1/2}\\ & & \times\left(\sum_{\gamma \in \wh{H_h}}{|\wh{\tau_{-h}(f_\mathcal{A})}(\gamma)|^2}\right)^{1/2}
\end{eqnarray*}
By Parseval's theorem
\begin{equation*}
\sum_{\gamma \in \wh{H_h}}{|\wh{1_{A_h'}}(\gamma)|^2}=\P_{H_h}(A_h') \textrm{ and }\sum_{\gamma \in \wh{H_h}}{|\wh{\tau_{-h}(f_\mathcal{A})}(\gamma)|^2}=\|f_\mathcal{A}\|_{L^2(h+H_h)}^2,
\end{equation*}
and combining all this with (\ref{eqn.qrness}) tells us that
\begin{eqnarray*}
\sum_{\gamma \neq 0_{\wh{H_h}}}{|\wh{1_{A_h'}}(\gamma)|^2|\wh{\tau_{-h}(f_\mathcal{A})}(\gamma)|}  & \leq & \epsilon \P_{H_h}(A_h')^{3/2}\|f_\mathcal{A}\|_{L^2(h+H_h)}\\ & \leq & \epsilon \P_{H_h}(A_h')^{3/2}\sqrt{2K} (f_\mathcal{A} \ast \P_{H_h})(h).
\end{eqnarray*}
The last inequality here follows from the fact that $h\in S_0$ ensures that $f_\mathcal{A}(h) \leq K\alpha$ and $(f_\mathcal{A} \ast \P_{H_h})(h) \geq \alpha/2$. Finally, our choice of $\epsilon$ tells us that
\begin{equation*}
\sum_{\gamma \neq 0_{\wh{H_h}}}{|\wh{1_{A_h'}}(\gamma)|^2|\wh{\tau_{-h}(f_\mathcal{A})}(\gamma)|}  \leq  \P_{H_h}(A_h')^{2}(f_\mathcal{A} \ast \P_{H_h})(h)/2,
\end{equation*}
whence, inserting this in (\ref{eqn.earlier}), we get that
\begin{equation*}
 \langle 1_{A_h'} \ast 1_{A_h'},\tau_{-h}(f_\mathcal{A}) \rangle_{L^2(H_h)}  \geq \P_{H_h}(A_h')^2(f_\mathcal{A}\ast \P_{H_h})(h)/2.
\end{equation*}
Thus, our earlier averaging tells us that
\begin{equation*}
\Lambda(\mathcal{A}) \geq \E_{h \in H}{1_{S_0}(h)\P_H(H_h)^2.\P_{H_h}(A_h')^2(f_\mathcal{A}\ast \P_{H_h})(h)/2},
\end{equation*}
and hence immediately that
\begin{eqnarray*}
\Lambda(\mathcal{A}) & \geq & 2^{-3}c^2\alpha\E_{h \in H}{1_{S_0}(h)\P_H(H_h)^2}\\ & = & \alpha^3\sigma \exp(-O((c^{-1}+K^{1/2}c^{-1/2})\log c^{-1})).
\end{eqnarray*}
The case is complete.
\end{proof}

\begin{case}
Suppose that $\P_H(S_1) \geq \sigma/2$. Then
\begin{equation*}
\Lambda(\mathcal{A}) \geq \exp(-O(L(\log^2\alpha^{-1}+\log \sigma^{-1})\exp(O((c^{-1}+K^{1/2}c^{-1/2})\log c^{-1})))).
\end{equation*}
\end{case}
\begin{proof}
Suppose that $h \in S_1$ so that $(f_\mathcal{A} \ast \P_{H_h})(h) \leq \alpha/2$. By the Fourier inversion formula we have
\begin{equation*}
 \sum_{ \gamma \in H_h^\perp}{\wh{f_\mathcal{A}}(\gamma)\gamma(h)} =(f_\mathcal{A} \ast \P_{H_h})(h)  \leq \alpha/2.
\end{equation*}
Separating out the trivial mode where $\wh{f_\mathcal{A}}(0_{\wh{H}})=\alpha$ and apply the triangle inequality we have that
\begin{equation*}
\sum_{0_{\wh{H}} \neq \gamma \in H_h^\perp}{|\wh{f_\mathcal{A}}(\gamma)|} \geq \alpha/2.
\end{equation*}
Write $\mathcal{L}':=\{\gamma : |\wh{f_\mathcal{A}}(\gamma)| \geq \P_H(H_h)\alpha/4\}$ and note that since $|H_h^\perp| = \P_{H}(H_h)^{-1}$ we have
\begin{equation*}
\sum_{0_{\wh{H}} \neq \gamma \in \mathcal{L}'\cap H_h^\perp}{|\wh{f_\mathcal{A}}(\gamma)|}  \geq \sum_{0_{\wh{H}} \neq \gamma \in H_h^\perp}{|\wh{f_\mathcal{A}}(\gamma)|}  -\sum_{\gamma \in H_h^\perp \setminus \mathcal{L}'}{|\wh{f_\mathcal{A}}(\gamma)|} \geq \alpha/4.
\end{equation*}
Since $\sup_{\gamma \neq 0_{\wh{H}}}{|\wh{f_\mathcal{A}}(\gamma)|} \leq L\alpha^2$ we conclude that
\begin{equation*}
|\mathcal{L}'\cap H_h^\perp| \geq \alpha^{-1}/4L.
\end{equation*}
Let $I_h \subset \mathcal{L}'\cap H_h^\perp$ be a set of $d:=\lfloor \log_2 ( \alpha^{-1}/4L)\rfloor$ independent elements -- possible since $2^d \leq |\mathcal{L}' \cap H_h^\perp|$ -- and put $H_h':=I_h^\perp$. Since $I_h \subset H_h^\perp$, it follows that $H_h'=I_h^\perp \supset H_h$, whence $H_h \leq H_h'$. Since the elements of $I_h$ are independent we have that
\begin{equation*}
\P_H(H_h') =|I_h|^{-1} = 2^{-d} \leq 8L \alpha.
\end{equation*}
Since $h \in S_1$ we also have
\begin{equation*}
\P_H(H_h) = \exp(-O((c^{-1}+K^{1/2}c^{-1/2})\log c^{-1}))\alpha,
\end{equation*}
whence
\begin{equation*}
|H_h':H_h| \leq  L\exp(O((c^{-1}+K^{1/2}c^{-1/2})\log c^{-1})),
\end{equation*}
and it follows that 
\begin{equation*}
\P_{H_h'}(A_h') \geq L^{-1}\exp(-O((c^{-1}+K^{1/2}c^{-1/2})\log c^{-1})).
\end{equation*}
Thus there is some $\delta$ with $\delta |H|$ an integer and
\begin{equation*}
\delta \geq L^{-1}\exp(-O((c^{-1}+K^{1/2}c^{-1/2})\log c^{-1}))
\end{equation*}
such that $\P_{H_h'}(A_h') \geq \delta$ for all $h \in S_1$; for each $h \in S_1$ let $A_h''$ be a subset of $A_h'$ of density $\delta$.

Each $H_h'$ is defined by the set $I_h$ and there are at most $\binom{|\mathcal{L}'|}{d}$ such sets since $I_h \subset \mathcal{L}'$. It follows that there is some space $H' \leq H$ such that $H_h'=H'$ for at least a proportion $\binom{|\mathcal{L}'|}{d}^{-1}$ of the elements of $S_1$; call this set $S_2$.

We now turn to estimating the density of $S_2$. First, by Parseval's theorem
\begin{equation*}
|\mathcal{L}'|(\P_H(H_h)\alpha/4)^2 \leq \sum_{\gamma \in \wh{H}}{|\wh{f_\mathcal{A}}(\gamma)|^2} =\|f_\mathcal{A}\|_{L^2(H)}^2 \leq K\alpha^2.
\end{equation*}
It follows from the lower bounds in $\P_H(H_h)$ that
\begin{equation*}
|\mathcal{L}'| \leq \alpha^{-2}\exp(O((c^{-1}+K^{1/2}c^{-1/2})\log c^{-1})),
\end{equation*}
whence
\begin{equation*}
\binom{|\mathcal{L}'|}{d} \leq \exp(O((c^{-1}+K^{1/2}c^{-1/2})\log^2\alpha^{-1}\log c^{-1})).
\end{equation*}
This tells us that
\begin{equation*}
\P_H(S_2) \geq \P_H(S_1)/\binom{|\mathcal{L}'|}{d} \geq \sigma\exp(-O((c^{-1}+K^{1/2}c^{-1/2})\log^2\alpha^{-1}\log c^{-1})).
\end{equation*}
Finally, by averaging let $h_1+H'$ be a coset of $H'$ on which $S_2$ has at least the above density and define a new family $\mathcal{A}'''$ on $H'$ as follows. For each $h' \in S_2-h_1$, let $A_{h'}''':=A_{h_1+h'}''$; if $h' \in H' \setminus (S_2-h_1)$ then let $A_{h'}''':=\emptyset$. By the definition of $S_2$ for each $h' \in S_2-h_1$ $A_{h_1+h'}''$ is a subset of $H'=H_{h_1+h'}'$ of density $\delta$. Thus by Proposition \ref{prop.largeL} we have
\begin{eqnarray*}
\Lambda(\mathcal{A}''') & = & \exp(-O(\delta^{-1}(\log^2\alpha^{-1}(c^{-1}+K^{1/2}c^{-1/2})\log c^{-1}+\log \sigma^{-1})))\\ & = & \exp(-O(L(\log^2\alpha^{-1}+\log \sigma^{-1})\exp(O((c^{-1}+K^{1/2}c^{-1/2})\log c^{-1})))).
\end{eqnarray*}
Finally it remains for us to check that $|H|^4\Lambda(\mathcal{A}) \geq |H'|^4\Lambda(\mathcal{A}''')$ from which the case follows; we proceed in the usual manner. 

There are $|H'|^4\Lambda(\mathcal{A}')$ quadruples $(a_0',a_1',y',h')$ with $a_0', a_1' \in A_{h'}'''$ and $y' \in A_{a_0'+a_1'-h'}'''$. Every such quadruple corresponds uniquely to a quadruple
\begin{equation*}
(a_0,a_1,y,h):=(a_0'+x_{h_1+h'},a_1'+x_{h_1+h'},y'+x_{a_0'+a_1'-h'+h_1},h_1+h')
\end{equation*}
with $a_0,a_1 \in A_h$ and $y \in A_{a_0+a_1-h}$, whence $|H'|^4\Lambda(\mathcal{A}) \leq |H|^4\Lambda(\mathcal{A})$ and the result follows.
\end{proof}
Having concluded both cases it remains to note that certainly one of $\P_H(S_1)$ and $\P_H(S_0)$ is at least $\sigma/2$ and so at least one of the cases occurs.
\end{proof}

\section{Families with small mean square density}\label{sec.fams}

In this section we use our previous work to establish the following lemma which is the main driver in the proof of Theorem \ref{thm.weighted} in the case when the density function has small mean square.
\begin{lemma}\label{lem.mainitlem}
Suppose that $H$ is a finite abelian group of exponent $2$, $\mathcal{A}=(A_h)_{h \in H}$ is a family on $H$ of density $\alpha$, $\|f_\mathcal{A}\|_{L^2(H)}^2=K\alpha^2$ and $L \geq \max\{K,2\}$ is a parameter. Then there is an absolute constant $C_{\mathcal{S}}>0$ such that either
\begin{equation*}
\Lambda(\mathcal{A}) \geq \exp(-(1+\log^2\alpha^{-1})\exp(C_{\mathcal{S}}L^3\log^2 L)))
\end{equation*}
or there is a subgroup $H' \leq H$ of index $2$ and a family $\mathcal{A}'$ on $H'$ such that
\begin{equation*}
\P_{H'}(\mathcal{A}') \geq \alpha+L\alpha^2/4K \textrm{ and } \Lambda(\mathcal{A}) \geq 2^{-4}\Lambda(\mathcal{A}').
\end{equation*}
\end{lemma}
\begin{proof}
Let $S_L:=\{h \in H:f_\mathcal{A}(h) \geq 4K\alpha\}$ and $S_S:=\{h \in H: f_\mathcal{A}(h) \leq \alpha/4\}$. Now, 
\begin{equation*}
\E_{h \in H}{1_{S_L}(h)f_\mathcal{A}(h)} \leq \frac{1}{4K\alpha}.\E_{h \in H}{1_{S_L}(h)f_\mathcal{A}(h)^2} \leq \frac{\alpha}{4},
\end{equation*}
and
\begin{equation*}
\E_{h \in H}{1_{S_S}(h)f_\mathcal{A}(h)} \leq \alpha/4
\end{equation*}
trivially, whence, putting $S:=H \setminus (S_L \cup S_S)$, we have that
\begin{equation*}
\E_{h \in H}{1_S(h)f_\mathcal{A}(h)} \geq \alpha/2.
\end{equation*}
Let $S_i:=\{h \in S: 2^{i-2}\alpha \leq f_\mathcal{A}(h) \leq 2^{i-1}\alpha\}$ and note that
\begin{equation*}
\sum_{i \leq \lceil \log K \rceil +1}{\E_{h \in H}{1_{S_i}(h).2^{i-1}\alpha}} \geq \alpha/2,
\end{equation*}
and thus by averaging there is some $i \leq \lceil \log K \rceil +1$ such that
\begin{equation*}
\E_{h \in H}{1_{S_i}(h).2^{i-1}\alpha} \geq \alpha/2(\lceil \log K \rceil +1).
\end{equation*}
As a by product note that $\P_H(S_i) = \Omega(1/K\log K)$ and we write $K_i = 2^{i-1}$ so that
\begin{equation*}
K_i \alpha \geq f_\mathcal{A}(h) \geq K_i\alpha/2 \textrm{ for all } h \in S_i
\end{equation*}
and
\begin{equation*}
\E_{h \in H}{1_{S_i}(h)f_\mathcal{A}(h)} = \Omega(\alpha/ (1+\log K)).
\end{equation*}

Now, if 
\begin{equation*}
\E_{h \in H}{1_{S_i}(h)|\wh{1_{A_h}}(\gamma)|^2} \geq L\alpha^3
\end{equation*}
then since $|\wh{1_{A_h}}(\gamma)| \leq 4K\alpha$ if $h \in S$ we conclude that
\begin{equation*}
\E_{h \in H}{1_{S_i}(h)|\wh{1_{A_h}}(\gamma)|} \geq L\alpha^2/4K.
\end{equation*}
Applying Lemma \ref{lem.coreincrement} we find we are in the second case of Lemma \ref{lem.mainitlem}.  Similarly, by Lemma \ref{lem.coreincrement2} we are done if $|\wh{f_\mathcal{A}}(\gamma)| \geq L\alpha^2/4K$. Thus we may assume that
\begin{equation}\label{eqn.asm1}
\sup_{\gamma \neq  0_{\wh{H}}}{\E_{h \in H}{1_{S_i}(h)|\wh{1_{A_h}}(\gamma)|^2} } \leq L\alpha^3,
\end{equation}
\begin{equation}\label{eqn.asm2}
\sup_{\gamma \neq  0_{\wh{H}}}{\E_{h \in H}{1_{S_i}(h)|\wh{1_{A_h}}(\gamma)|}} \leq L\alpha^2/4K
\end{equation}
and
\begin{equation}\label{eqn.asm3}
\sup_{\gamma \neq  0_{\wh{H}}}{|\wh{f_\mathcal{A}}(\gamma)|} \leq L\alpha^2/4K.
\end{equation}
As usual, by non-negativity of the terms in $\Lambda(\mathcal{A})$ we have that
\begin{equation*}
\Lambda(\mathcal{A}) \geq \E_{h \in H}{1_{S_i}(h)\langle \tau_h(1_{A_h} \ast 1_{A_h}),f_\mathcal{A} \rangle_{L^2(H)}}.
\end{equation*}
We apply Plancherel's theorem to the inner products on the right to get
\begin{equation*}
\langle \tau_h(1_{A_h} \ast 1_{A_h}),f_\mathcal{A} \rangle_{L^2(H)}=\sum_{\gamma \in \wh{H}}{|\wh{1_{A_h}}(\gamma)|^2\wh{f_\mathcal{A}}(\gamma)\gamma(h)}.
\end{equation*}
Separating out the trivial mode and applying the triangle inequality then tells us that
\begin{equation*}
\langle \tau_h(1_{A_h} \ast 1_{A_h}),f_\mathcal{A} \rangle_{L^2(H)} \geq f_\mathcal{A}(h)^2\alpha - \sum_{\gamma \neq 0_{\wh{H}}}{|\wh{1_{A_h}}(\gamma)|^2|\wh{f_\mathcal{A}}(\gamma)|}.
\end{equation*}
Thus
\begin{equation*}
\sum_{\gamma \neq 0_{\wh{H}}}{\E_{h \in H}{1_{S_i}(h)|\wh{1_{A_h}}(\gamma)|^2}|\wh{f_\mathcal{A}}(\gamma)|} \geq \alpha\E_{h \in H}{1_{S_i}(h)f_\mathcal{A}(h)^2} - \Lambda(\mathcal{A}).
\end{equation*}
It follows that either
\begin{equation*}
\Lambda(\mathcal{A}) \geq \alpha\E_{h \in H}{1_{S_i}(h)f_\mathcal{A}(h)^2}/2 =\Omega(\alpha^3/(1+\log K)),
\end{equation*}
and we are done or
\begin{equation}\label{eqn.rop}
\sum_{\gamma \neq 0_{\wh{H}}}{\E_{h \in H}{1_{S_i}(h)|\wh{1_{A_h}}(\gamma)|^2}|\wh{f_\mathcal{A}}(\gamma)|} \geq \alpha\E_{h \in H}{1_{S_i}(h)f_\mathcal{A}(h)^2}/2,
\end{equation}
which we now assume. Let 
\begin{equation*}
\mathcal{L}:=\left\{\gamma \in \wh{H}:\E_{h \in H}{1_{S_i}(h)|\wh{1_{A_h}}(\gamma)|^2} \geq \frac{(\E_{h \in H}{1_{S_i}(h)f_\mathcal{A}(h)^2})^2}{2^4K\E_{h \in H}{1_{S_i}(h)f_\mathcal{A}(h)}}\right\}.
\end{equation*}
It is easy enough to see that
\begin{equation}\label{eqn.cla}
\sum_{\gamma \not\in \mathcal{L}}{\E_{h \in H}{1_{S_i}(h)|\wh{1_{A_h}}(\gamma)|^2}|\wh{f_\mathcal{A}}(\gamma)|} \leq \alpha\E_{h \in H}{1_{S_i}(h)f_\mathcal{A}(h)^2}/4
\end{equation}
as we shall now show. We apply the triangle inequality to the left hand side after swapping the order of summation to get that it is at most
\begin{equation*}
\sup_{\gamma \not \in \mathcal{L}}{(\E_{h \in H}{1_{{S_i}}(h) |\wh{1_{A_h}}(\gamma)|^2})^{1/2}}\sum_{\gamma \in \wh{H}}{(\E_{h \in H}{1_{{S_i}}(h) |\wh{1_{A_h}}(\gamma)|^2})^{1/2}|\wh{f_\mathcal{A}}(\gamma)|}.
\end{equation*}
Now apply the Cauchy-Schwarz inequality to this to see that the sum is at most
\begin{equation*}
\left(\sum_{\gamma \in \wh{H}}{\E_{h \in H}{1_{{S_i}}(h) |\wh{1_{A_h}}(\gamma)|^2}}\right)^{1/2}\left(\sum_{\gamma \in \wh{H}}{|\wh{f_\mathcal{A}}(\gamma)|^2}\right)^{1/2} = \sqrt{\E_{h \in H}{1_{{S_i}}(h)f_\mathcal{A}(h)}K\alpha^2}
\end{equation*}
by Parseval's theorem after interchanging the order of summation again. The bound (\ref{eqn.cla}) now follows from the definition of $\mathcal{L}$. Combining this with (\ref{eqn.rop}) we see that
\begin{eqnarray*}
\sum_{0_{\wh{H}}\neq \gamma \in \mathcal{L}}{\E_{h \in H}{1_{S_i}(h)|\wh{1_{A_h}}(\gamma)|^2}|\wh{f_\mathcal{A}}(\gamma)|} & \geq & \alpha\E_{h \in H}{1_{S_i}(h)f_\mathcal{A}(h)^2}/4\\ & = & \Omega(K_i\alpha^3/(1+\log K)).
\end{eqnarray*}
Write
\begin{equation*}
\mathcal{L}_j:=\{\gamma \in \wh{H}:2^{-j}L\alpha^3 \geq \E_{h \in H}{1_{S_i}(h)|\wh{1_{A_h}}(\gamma)|^2} \geq 2^{-(j+1)}L\alpha^3\},
\end{equation*}
and note that by (\ref{eqn.asm1}) we have that $\mathcal{L}\setminus\{0_{\wh{H}}\}=\bigcup_{j=0}^{j_0}{\mathcal{L}_j}$, where $j_0$ is the smallest integer such that
\begin{equation*}
2^{-(j_0+1)}L\alpha^3\leq (\E_{h \in H}{1_{S_i}(h)f_\mathcal{A}(h)^2})^2/2^4K\E_{h \in H}{1_{S_i}(h)f_\mathcal{A}(h)};
\end{equation*}
crucially
\begin{equation*}
j_0=O(\log L) \textrm{ and } 2^{-(j_0+1)}L\alpha^3 = \Omega(K_i^2\alpha^3/K(1+\log K)).
\end{equation*}
It follows by averaging (and since $L \geq \max\{2,K\}$) that there is some $j\leq j_0$ such that
\begin{equation*}
\sum_{0_{\wh{H}}\neq \gamma \in \mathcal{L}_j}{\E_{h \in H}{1_{S_i}(h)|\wh{1_{A_h}}(\gamma)|^2}|\wh{f_\mathcal{A}}(\gamma)|} =\Omega(K_i\alpha^3/(1+\log K)\log L).
\end{equation*}
Inserting (\ref{eqn.asm3}) and dividing gives that
\begin{equation*}
\sum_{0_{\wh{H}}\neq \gamma \in \mathcal{L}_j}{\E_{h \in H}{1_{S_i}(h)|\wh{1_{A_h}}(\gamma)|^2}}=\Omega(K_iK\alpha/(1+\log K)L\log L).
\end{equation*}

Now, the usual convexity of $L^p$-norms tells us that
\begin{equation*}
\E_{h \in H}{1_{S_i}(h)|\wh{1_{A_h}}(\gamma)|^2} \leq \left(\E_{h \in H}{1_{S_i}(h)|\wh{1_{A_h}}(\gamma)|}\right)^{2/3}\left(\E_{h \in H}{1_{S_i}(h)|\wh{1_{A_h}}(\gamma)|^4}\right)^{1/3}.
\end{equation*}
Thus, by (\ref{eqn.asm2}) we have
\begin{equation*}
\left(\E_{h \in H}{1_{S_i}(h)|\wh{1_{A_h}}(\gamma)|^2}\right)^3 \leq \frac{L^2\alpha^4}{2^4K^2}\E_{h \in H}{1_{S_i}(h)|\wh{1_{A_h}}(\gamma)|^4}.
\end{equation*}
Dividing out and summing over $\mathcal{L}_j$, using the fact that it is a dyadic range, tells us that
\begin{equation*}
\sum_{\gamma \in \wh{H}}{\E_{h \in H}{1_{S_i}(h)|\wh{1_{A_h}}(\gamma)|^4}} = \Omega(\alpha^3 K_i^5K/(1+\log K)^3L^3 \log L).
\end{equation*}
Thus Parseval's theorem reveals that
\begin{equation*}
\E_{h \in H}{1_{S_i}(h)}{\|1_{A_h} \ast 1_{A_h}\|_{L^2(H)}^2} = \Omega(\alpha^3 K_i^5K/(1+\log K)^3L^3 \log L).
\end{equation*}
Finally, let
\begin{equation*}
S_i':=\{h \in S_i: \|1_{A_h} \ast 1_{A_h}\|_{L^2(H)}^2\geq \E_{h \in H}{1_{S_i}(h)}{\|1_{A_h} \ast 1_{A_h}\|_{L^2(H)}^2} /2\}
\end{equation*}
 and note that if $h \in S_i'$ then $f_\mathcal{A}(h) \leq K_i\alpha$, whence
 \begin{equation*}
 \|1_{A_h} \ast 1_{A_h}\|_{L^2(H)}^2=\Omega(\alpha^3 K_i^2K/(1+\log K)^3L^3 \log L).
 \end{equation*}
 Furthermore
\begin{equation*}
\E_{h \in H}{1_{S_i'}(h)}{\|1_{A_h} \ast 1_{A_h}\|_{L^2(H)}^2} \geq \Omega(\alpha^3 K_i^5K/(1+\log K)^3L^3 \log L),
\end{equation*}
whence $\P_H(S) =\Omega(L^4)$. We now apply Lemma \ref{prop.highnrg} to see that
\begin{equation*}
\Lambda(\mathcal{A}) \geq \exp(-(1+\log^2\alpha^{-1})\exp(O(K^{-1}(1+\log K)^3L^3\log L)))).
\end{equation*}
However, $K^{-1}(1+\log K)^3=O(1)$, whence we get the result.
\end{proof}
It may seem bizarre to have thrown away the extra strength of the $K^{-1}(1+\log K)^3$-term at the very end of the above proof. However, in applications we shall have a dichotomy between the case when $K$ is large and when $K$ is small. In the latter we shall not, in fact, be able to guarantee that $K$ is much bigger than $1$ whence the above estimate of $K^{-1}(1+\log K)^3=O(1)$ is tight.

\section{Proof of Theorem \ref{thm.weighted}}\label{sec.end}

As will have become clear the proof of Theorem \ref{thm.weighted} is iterative and is driven by Lemma \ref{lem.mainitlem} and Corollary \ref{cor.kuy}.
\begin{proof}[Proof of Theorem \ref{thm.weighted}]
Let $H_0:=\Im 2$ and $\mathcal{A}_0$ be the family corresponding to the set $A$, which has density $\alpha_0=\alpha$. We shall define a sequence of families $(\mathcal{A}_i)_i$ on subgroups $(H_i)_i$ with density $\alpha_i$ and the following properties:
\begin{equation*}
\Lambda(\mathcal{A}_{i+1}) \leq 2^{-4}\Lambda(\mathcal{A}_i) \leq 2^{-4i}\Lambda(A)
\end{equation*}
and
\begin{equation*}
\alpha_{i+1} \geq \alpha_i(1+ \Omega(\alpha_i \log^{1/6} \alpha_i^{-1}\log\log^{-5/3}\alpha_i^{-1})).
\end{equation*}
It is useful to define the auxiliary variables $K_i$ and $L_i$: let $L_i$ be the solution to 
\begin{equation*}
C_\mathcal{S}L_i^3\log^2L_i = \log \alpha_i^{-1}/2 \textrm{ and }K_i:=\alpha_i^{-2}\|f_{\mathcal{A}_i}\|_{L^2(H_i)}^2.
\end{equation*}
Suppose that we are at stage $i$ of the iteration; we consider two cases:
\begin{enumerate}
\item If $L_i \leq 2+K_i^2/(1+\log K_i)^2$ then apply Corollary \ref{cor.kuy} and terminate the iteration with
\begin{eqnarray*}
\Lambda(\mathcal{A}_i) & = & \exp(-O(\alpha_i^{-1}K_i^{-1}\log^2 K_i))\\ & = & \exp(-O(\alpha^{-1}\log^{-1/6} \alpha^{-1}\log \log^{5/3} \alpha^{-1})).
\end{eqnarray*}
\item If $L_i>2+K_i^2/(1+\log K_i)^2$ then apply Lemma \ref{lem.mainitlem} with parameter $L_i$. If we have the first conclusion of the lemma then 
\begin{equation*}
\Lambda(\mathcal{A}_i) \geq \exp(-(1+\log\alpha_i^{-1})^2\exp(C_{\mathcal{S}}L_i^3\log^2 L_i)).
\end{equation*}
In view of the definition of $L_i$ and the fact that $\alpha_i \geq \alpha$ we conclude that $\exp(C_{\mathcal{S}}L_i^3\log^2 L_i)\leq \alpha^{-1/2}$, whence we certainly have
\begin{equation*}
\Lambda(\mathcal{A}_i) =\exp(-O(\alpha^{-1}\log^{-1/6} \alpha^{-1}\log \log^{5/3} \alpha^{-1}))
\end{equation*}
again. The other conclusion of Lemma \ref{lem.mainitlem} tells us that we have a new subgroup $H_{i+1} \leq H_i$, and a family $\mathcal{A}_{i+1}$ on $H_{i+1}$ with
\begin{equation*}
\alpha_{i+1} \geq \alpha_i(1+ (L_i/4K_i)\alpha^{-1}) \textrm{ and } \Lambda(\mathcal{A}_{i+1}) \geq 2^{-4}\Lambda(\mathcal{A}_i);
\end{equation*}
this has the desired property for the iteration.
\end{enumerate}
In view of the lower bound on $\alpha_i$ we see that the density doubles in
\begin{equation*}
F(\alpha)=O(\alpha^{-1}\log^{-1/6}\alpha^{-1}\log \log^{5/3}\alpha^{-1})
\end{equation*}
steps, whence the iteration must terminate in at most $F(\alpha)+F(2\alpha)+F(2^2\alpha)+\dots$ steps. Of course $F(2\alpha') \leq F(\alpha')/\sqrt{2}$ whenever $\alpha' \in (0, c_0]$ for some absolute constant $c_0$. Thus, on summing the geometric progression we see that the iteration terminates in $O(F(\alpha))$ steps. It follows that at the time of termination we have
\begin{equation*}
\Lambda(A) \geq \exp(-O(\alpha^{-1}\log^{-1/6}\alpha^{-1}\log \log^{5/3}\alpha^{-1}))\Lambda(\mathcal{A}_i),
\end{equation*}
and we get the result.
\end{proof}

\section{Concluding remarks}\label{sec.obstacles}

No doubt some improvement could be squeezed out of our arguments by more judicious averaging but there is a natural limit placed on the method by Corollary \ref{cor.bsgcor} and it seems that to move the $1/6$ in Theorem \ref{thm.weighted} past $1$ would require a new idea. This, however, is a little frustrating for the following reason.

The well-known Erd{\H o}s-Tur{\'a}n conjecture is essentially equivalent to asking for Roth's theorem in $\Z/N\Z$ for any set of density $\delta(N)$ where $\delta(N)$ is a function with $\sum_{N}{N^{-1}\delta(N)} = \infty$. In particular, $\delta(N) = 1/\log N\log\log N\log\log\log N$ satisfies this hypothesis and so to have the analogue of the Erd{\H o}s-Tur{\'a}n conjecture in $\Z_4^n$ we would need to push the constant $1/6$ past $1$.

In light of the heuristic in \S\ref{sec.out} one might reasonably conjecture the following much stronger result.
\begin{conjecture}
Suppose that $G=\Z_4^n$ and $A \subset G$ contains no proper three-term arithmetic progressions. Then $|A|=O(|G|/\log^{3/2} |G|)$.
\end{conjecture}
Of course it may well be that much more is true. We were able to find the following lower bound; as with $\Z_3^n$ (where the best lower bound is due to Edel \cite{YE}, but see also \cite{YLJW}) its density is of power shape.
\begin{proposition}\label{prop.tr}
Suppose that $G=\Z_4^n$. Then there is a set $A \subset G$ with no proper three-term arithmetic progressions and $|A|=\Omega(|G|^{2/3})$.
\end{proposition}
\begin{proof}
Note that the set $A_0$ containing the elements
\begin{equation*}
\begin{array}{cccc}
(0,0,0) & (0,0,1) & (0,1,0) & (0,1,2)\\
(0,2,1) & (0,2,2) & (1,0,0) & (1,0,2)\\
(1,2,0) & (1,2,2) & (2,0,1) & (2,0,2)\\
(2,1,0) & (2,1,2) & (2,2,0) & (2,2,1)\\
\end{array}
\end{equation*}
is a set of size $16$ in $\Z_4^3$ which contains no proper three-term arithmetic progressions. The result now follows on noting that the product of two sets not containing any proper three-term arithmetic progressions does, itself, not contain any proper three-term arithmetic progressions:

Suppose that $B$ and $C$ are such sets and $(x_0,x_1),(y_0,y_1),(z_0,z_1) \in B\times C$ have $x+y=2.z$. Then $x_i+y_i=2.z_i$ for $i \in \{0,1\}$. However since $B$ and $C$ do not contain any proper progressions we have that $x_i=y_i$ for all $i \in \{0,1\}$ whence $x=y$ and so the progression is not proper.
\end{proof}
We are unaware of any serious search for better choices of $A_0$, although such no doubt exist.  Indeed,  recently Elsholtz observed that a more general construction designed for Moser's cube problem may be used.

Moser asked for large subsets of $\{0,1,2\}^n$ not containing three points on a line; Koml{\'o}s and Chv{\'a}tal \cite{VC} note that the sets
\begin{equation*}
S_n:=\{x \in \{0,1,2\}^n: x_i=1 \textrm{ for } \lfloor n/3 \rfloor \textrm{ values of } i \in [n]\},
\end{equation*}
have size $\Omega(3^n/\sqrt{n})$ by Stirling's formula and satisfy Moser's requirement.  Our set $A_0$ is equal to $S_3$.  Embedding $S_n$ in $\Z_4^n$ in the obvious way it may be checked that the lack of lines in $S_n$ yields a set containing no proper three-term arithmetic progressions and hence the following theorem.
\begin{theorem}[{\cite[Theorem 3]{CE}}] 
Suppose that $G=\Z_4^n$. Then there is a set $A \subset G$ with no proper three-term arithmetic progressions and
\begin{equation*}
|A|=\Omega(|G|^{\log 3/\log 4}/\sqrt{\log |G|}).
\end{equation*}
\end{theorem}
The reader may wish to know that $\log 3/\log 4 = 0.792\dots$. The details along with some other results and generalisations are supplied in Elsholtz's paper.

\section*{Acknowledgments}
The author would like to thank Ernie Croot for a number of very useful conversations, Christian Elsholtz for supplying the preprint \cite{CE}, Olof Sisask for writing a program to find the example in Proposition \ref{prop.tr}, Terry Tao for useful comments and two anonymous referees for useful comments and careful reading.

\bibliographystyle{alpha}

\bibliography{master}

\end{document}